\documentclass[11pt,twoside]{amsart}
\usepackage{xcolor}
\usepackage{graphics}
\usepackage{color}
\usepackage[a4paper, margin=3cm]{geometry}
\usepackage{booktabs}
\usepackage{footnote}
\usepackage{amssymb,enumerate}
\usepackage{amsmath}
\usepackage{bbm}           
\usepackage{bm}
\usepackage{eso-pic,graphicx}
\usepackage{tikz}
\usepackage{cite}
\usepackage{esint}
\usepackage[colorlinks=true, pdfstartview=FitV, linkcolor=blue, citecolor=blue, urlcolor=blue]{hyperref}
\usepackage{booktabs}
\usepackage{graphicx}
\usepackage{pslatex}
\usepackage{amsmath,amsfonts}
\usepackage{rotating}
\usepackage{amssymb}
\usepackage{verbatim}
\usepackage{rotating}
\usepackage{mathrsfs}

\makeatletter
\def\Ddots{\mathinner{\mkern1mu\raise\p@
		\vbox{\kern7\p@\hbox{.}}\mkern2mu
		\raise4\p@\hbox{.}\mkern2mu\raise7\p@\hbox{.}\mkern1mu}}
\makeatother

\def\XXint#1#2#3{{\setbox0=\hbox{$#1{#2#3}{\int}$}
		\vcenter{\hbox{$#2#3$}}\kern-.5\wd0}}

\newtheorem*{thA}{Theorem A}
\newtheorem*{thA.1}{Theorem A.1}
\newtheorem*{thA.2}{Theorem A.2}
\newtheorem*{thB}{Theorem B}
\newtheorem*{thB.1}{Theorem B.1}
\newtheorem*{thB.2}{Theorem B.2}

\def\XXint#1#2#3{{\setbox0=\hbox{$#1{#2#3}{\int}$}
		\vcenter{\hbox{$#2#3$}}\kern-.5\wd0}}

\begin{document}
	\newtheorem{theorem}{Theorem}
	\newtheorem{proposition}[theorem]{Proposition}
	\newtheorem{conjecture}[theorem]{Conjecture}
	\def\theconjecture{\unskip}
	\newtheorem{corollary}[theorem]{Corollary}
	\newtheorem{lemma}[theorem]{Lemma}
	\newtheorem{claim}[theorem]{Claim}
	\newtheorem{sublemma}[theorem]{Sublemma}
	\newtheorem{observation}[theorem]{Observation}
	\theoremstyle{definition}
	\newtheorem{definition}{Definition}
	\newtheorem{notation}[definition]{Notation}
	\newtheorem{remark}[definition]{Remark}
	\newtheorem{question}[definition]{Question}
	\newtheorem{questions}[definition]{Questions}
	\newtheorem{example}[definition]{Example}
	\newtheorem{problem}[definition]{Problem}
	\newtheorem{exercise}[definition]{Exercise}
	\newtheorem{thm}{Theorem}
	\newtheorem{cor}[thm]{Corollary}
	\newtheorem{lem}{Lemma}[section]
	\newtheorem{prop}[thm]{Proposition}
	\theoremstyle{definition}
	\newtheorem{dfn}[thm]{Definition}
	\theoremstyle{remark}
	\newtheorem{rem}{Remark}
	\newtheorem{ex}{Example}
	\numberwithin{equation}{section}
	\def\C{\mathbb{C}}
	\def\R{\mathbb{R}}
	\def\Rn{{\mathbb{R}^n}}
	\def\Rns{{\mathbb{R}^{n+1}}}
	\def\Sn{{{S}^{n-1}}}
	\def\M{\mathbb{M}}
	\def\N{\mathbb{N}}
	\def\Q{{\mathbb{Q}}}
	\def\Z{\mathbb{Z}}
	\def\F{\mathcal{F}}
	\def\L{\mathcal{L}}
	\def\S{\mathcal{S}}
	\def\supp{\operatorname{supp}}
	\def\essi{\operatornamewithlimits{ess\,inf}}
	\def\esss{\operatornamewithlimits{ess\,sup}}
	
	\numberwithin{equation}{section}
	\numberwithin{thm}{section}
	\numberwithin{theorem}{section}
	\numberwithin{definition}{section}
	\numberwithin{equation}{section}
	
	\def\earrow{{\mathbf e}}
	\def\rarrow{{\mathbf r}}
	\def\uarrow{{\mathbf u}}
	\def\varrow{{\mathbf V}}
	\def\tpar{T_{\rm par}}
	\def\apar{A_{\rm par}}
	
	\def\reals{{\mathbb R}}
	\def\torus{{\mathbb T}}
	\def\scriptm{{\mathcal T}}
	\def\heis{{\mathbb H}}
	\def\integers{{\mathbb Z}}
	\def\z{{\mathbb Z}}
	\def\naturals{{\mathbb N}}
	\def\complex{{\mathbb C}\/}
	\def\distance{\operatorname{distance}\,}
	\def\support{\operatorname{support}\,}
	\def\dist{\operatorname{dist}\,}
	\def\Span{\operatorname{span}\,}
	\def\degree{\operatorname{degree}\,}
	\def\kernel{\operatorname{kernel}\,}
	\def\dim{\operatorname{dim}\,}
	\def\codim{\operatorname{codim}}
	\def\trace{\operatorname{trace\,}}
	\def\Span{\operatorname{span}\,}
	\def\dimension{\operatorname{dimension}\,}
	\def\codimension{\operatorname{codimension}\,}
	\def\nullspace{\scriptk}
	\def\kernel{\operatorname{Ker}}
	\def\ZZ{ {\mathbb Z} }
	\def\p{\partial}
	\def\rp{{ ^{-1} }}
	\def\Re{\operatorname{Re\,} }
	\def\Im{\operatorname{Im\,} }
	\def\ov{\overline}
	\def\eps{\varepsilon}
	\def\lt{L^2}
	\def\diver{\operatorname{div}}
	\def\curl{\operatorname{curl}}
	\def\etta{\eta}
	\newcommand{\norm}[1]{ \|  #1 \|}
	\def\expect{\mathbb E}
	\def\bull{$\bullet$\ }
	
	\def\blue{\color{blue}}
	\def\red{\color{red}}
	
	\def\xone{x_1}
	\def\xtwo{x_2}
	\def\xq{x_2+x_1^2}
	\newcommand{\abr}[1]{ \langle  #1 \rangle}

	\newcommand{\Norm}[1]{ \left\|  #1 \right\| }
	\newcommand{\set}[1]{ \left\{ #1 \right\} }
	\newcommand{\ifou}{\raisebox{-1ex}{$\check{}$}}
	\def\one{\mathbf 1}
	\def\whole{\mathbf V}
	\newcommand{\modulo}[2]{[#1]_{#2}}
	\def \essinf{\mathop{\rm essinf}}
	\def\scriptf{{\mathcal F}}
	\def\scriptg{{\mathcal G}}
	\def\scriptm{{\mathcal M}}
	\def\scriptb{{\mathcal B}}
	\def\scriptc{{\mathcal C}}
	\def\scriptt{{\mathcal T}}
	\def\scripti{{\mathcal I}}
	\def\scripte{{\mathcal E}}
	\def\scriptv{{\mathcal V}}
	\def\scriptw{{\mathcal W}}
	\def\scriptu{{\mathcal U}}
	\def\scriptS{{\mathcal S}}
	\def\scripta{{\mathcal A}}
	\def\scriptr{{\mathcal R}}
	\def\scripto{{\mathcal O}}
	\def\scripth{{\mathcal H}}
	\def\scriptd{{\mathcal D}}
	\def\scriptl{{\mathcal L}}
	\def\scriptn{{\mathcal N}}
	\def\scriptp{{\mathcal P}}
	\def\scriptk{{\mathcal K}}
	\def\frakv{{\mathfrak V}}
	\def\C{\mathbb{C}}
	\def\D{\mathcal{D}}
	\def\R{\mathbb{R}}
	\def\Rn{{\mathbb{R}^n}}
	\def\rn{{\mathbb{R}^n}}
	\def\Rm{{\mathbb{R}^{2n}}}
	\def\r2n{{\mathbb{R}^{2n}}}
	\def\Sn{{{S}^{n-1}}}
	\def\M{\mathbb{M}}
	\def\N{\mathbb{N}}
	\def\Q{{\mathcal{Q}}}
	\def\Z{\mathbb{Z}}
	\def\F{\mathcal{F}}
	\def\L{\mathcal{L}}
	\def\G{\mathscr{G}}
	\def\ch{\operatorname{ch}}
	\def\supp{\operatorname{supp}}
	\def\dist{\operatorname{dist}}
	\def\essi{\operatornamewithlimits{ess\,inf}}
	\def\esss{\operatornamewithlimits{ess\,sup}}
	\def\dis{\displaystyle}
	\def\dsum{\displaystyle\sum}
	\def\dint{\displaystyle\int}
	\def\dfrac{\displaystyle\frac}
	\def\dsup{\displaystyle\sup}
	\def\dlim{\displaystyle\lim}
	\def\bom{\Omega}
	\def\om{\omega}

	\author[B. Dan]{BinWei Dan}
	\address{BinWei Dan:
		School of Mathematical Sciences \\
		Beijing Normal University \\
		Laboratory of Mathematics and Complex Systems \\
		Ministry of Education \\
		Beijing 100875 \\
		People's Republic of China}
	\email{bwdan@mail.bnu.edu.cn}
	
	\author[Q. Xue]{Qingying Xue$^{*}$}
	\address{Qingying Xue:
		School of Mathematical Sciences \\
		Beijing Normal University \\
		Laboratory of Mathematics and Complex Systems \\
		Ministry of Education \\
		Beijing 100875 \\
		People's Republic of China}
	\email{qyxue@bnu.edu.cn}
	
	
	\thanks{The authors were partly supported by the National Key R\&D Program of China (No. 2020YFA0712900) and NNSF of China (No. 12271041).
		\thanks{$^{*}$ Corresponding author, e-mail address: qyxue@bnu.edu.cn}}

	\date{\today}
	\title[ SPARSE DOMINATION FOR ROUGH MULTILINEAR SINGULAR INTEGRAL OPERATORS]
	{\bf Sparse Domination for Rough Multilinear Singular Integrals}

	\begin{abstract}

		Let \( \Omega \) be a function on \( \mathbb{R}^{mn} \), homogeneous of degree zero, and satisfy a cancellation condition on the unit sphere \( \mathbb{S}^{mn-1} \). In this paper, we show that the multilinear singular integral operator  
		\[
		\mathcal{T}_{\Omega}(f_1, \ldots, f_m)(x) := \mathrm{p.v.} \int_{\mathbb{R}^{mn}} \frac{\Omega(x - y_1, \ldots, x - y_m)}{|x - \vec{y}|^{mn}} \prod_{i=1}^m f_i(y_i) \, d\vec{y},
		\]
		associated with a rough kernel \( \Omega \in L^r(\mathbb{S}^{mn-1}) \), \( r > 1 \), admits a {sparse domination}, where $\quad \vec{y}=(y_1,\ldots,y_m)$ and $ d\vec{y}=dy_1\cdots dy_m$. As a consequence, we derive some {quantitative  weighted norm inequalities} for \( \mathcal{T}_{\Omega} \).

	\end{abstract}\maketitle
	
	\section{Introduction}

	Singular integral operators of convolution type
	\[
	Tf(x) := \mathrm{p.v.} \int_{\mathbb{R}^n} f(y) K(x - y) \, dy
	\]
	have long served as central objects of Harmonic analysis, originating from the celebrated work of Calder\'on and Zygmund~\cite{calderon_existence_1952}. For the class of homogeneous kernels of the form \( K(x) = \Omega(x/|x|) |x|^{-n} \), including the Hilbert and Riesz transforms, boundedness on \( L^p(\mathbb{R}^n) \) for \( 1 < p < \infty \) holds even under minimal smoothness assumptions on \( \Omega \). Initial results~\cite{calderon_singular_1956} required that \( \Omega \in L \log L(\mathbb{S}^{n-1}) \) and satisfy a cancellation condition on the unit sphere. Subsequent refinements~\cite{coifman_extensions_1977, connett_singular_1979} incorporated Hardy space techniques, while endpoint bounds were later developed in~\cite{christ_weak_1988, hofmann_weighted_1988, seeger_singular_1996}.
	
	Recent efforts have extended these ideas to the multilinear setting. Given \( m \in \mathbb{N} \) and a function \( \Omega \in L^q(\mathbb{S}^{mn-1}) \) with \( \int_{\mathbb{S}^{mn-1}} \Omega = 0 \), define the kernel
	\[
	K(x) := \Omega(x') |x|^{-mn}, \qquad x \in \mathbb{R}^{mn} \setminus \{0\}, \quad x' := \frac{x}{|x|},
	\]
	and consider the multilinear singular integral operator
\begin {equation} \label{1.1}
\mathcal{T}_{\Omega}(f_1, \ldots, f_m)(x) := \mathrm{p.v.} \int_{\mathbb{R}^{mn}} K(x - y_1, \ldots, x - y_m) \prod_{i=1}^m f_i(y_i) \, d\vec{y}, \quad d\vec{y}=dy_1 \cdots dy_m.
\end{equation}
	It was well known that the bilinear case \( m=2 \) was studied by Coifman and Meyer~\cite{coifman_commutators_1975}. More recently, Grafakos, He, and Honzík~\cite{grafakos_rough_2018} established the $L^{p_1}(\mathbb{R}^{n})\times L^{p_2}(\mathbb{R}^{n})$ to $L^p(\mathbb{R}^{n})$ boundedness of \( \mathcal{T}_{\Omega}\) when \( \Omega \in L^\infty(\mathbb{S}^{2n-1}) \), \(1/p_1+1/p_2=1/p\) and \( 1 < p_1, p_2 < \infty \). Their later work~\cite{grafakos_l2times_2020} lowered the regularity threshold to \( \Omega \in L^q(\mathbb{S}^{2n-1}) \) for \( q > 4/3 \), further refined by~\cite{he_improved_2023}. Most recently, Dosidis and Slavíková~\cite{dosidis_multilinear_2024} extended the above results to the entire range \( q > 1 \) for general multilinear operators $\mathcal{T}_\Omega$.
	
		To state the results in ~\cite{dosidis_multilinear_2024}, we recall some useful notations. For \( \alpha \in \{0,1\}^m \), let \( |\alpha| := \sum_{i=1}^m \alpha_i \). Define \( 1/p_\alpha := \sum_{i=1}^m \alpha_i/p_i \), and in particular \( 1/p := 1/p_{(1,\ldots,1)} \). Given \( r > 1 \), we write \( (\frac{1}{p_1},\cdots,\frac{1}{p_m}) \in \mathcal{H}^m(r) \) if
$\frac{1}{p_\alpha} + \frac{|\alpha| - 1}{r} < |\alpha|$ for all $\alpha \in \{0,1\}^m.$ The main result in \cite{dosidis_multilinear_2024} can be formulated as follows:
	\begin{thA}[\cite{dosidis_multilinear_2024}]\label{A}
	Let $1<p_1,\cdots,p_m<\infty$, $\frac{1}{p}=\sum^{m}_{i=1}\frac{1}{p_i}$ be such that $(\frac{1}{p_1},\cdots,\frac{1}{p_m})\in\mathcal{H}^m(r)$ and suppose that $\Omega\in L^r(\mathbb{S}^{mn-1})$, $r>1$, with $\int_{\mathbb{S}^{mn-1}}\Omega(\theta)d\sigma(\theta)=0$. Then there exists a constant $C=C(n,p_1,\cdots,p_m,r)$ such that
	$$\Vert \mathcal{T}_{\Omega}(f_1,\cdots,f_m)\Vert_{L^{p}(\mathbb{R}^{n})}\lesssim_{n,p_1,p_2,\cdots,p_m,r} \Vert\Omega\Vert_{L^r(\mathbb{S}^{mn-1})}\prod^{m}_{i=1}\Vert f_i\Vert_{L^{p_i}(\mathbb{R}^n)}.$$
\end{thA}

Consider now the sparse domination theory for $T_\Omega$. 	It was well known that sparse domination plays an important role in Harmonic analysis, particularly in deriving sharp weighted inequalities (e.g., the resolution of the \( A_2 \) conjecture~\cite{lacey_elementary_2017, lerner_simple_2013}). Recall that, given a constant \( 0 < \eta < 1 \), a collection \( \mathcal{S} \) of cubes in \( \mathbb{R}^n \) is called \( \eta \)-sparse if for each \( Q \in \mathcal{S} \) there exists a measurable subset \( E_Q \subset Q \) with \( |E_Q| \geq \eta |Q| \), and the sets \( \{E_Q\}_{Q \in \mathcal{S}} \) are pairwise disjoint.
	Given a sparse family \( \mathcal{S} \), define the positive sparse form
	\[
	\mathrm{PSF}_{\mathcal{S}; p_1, \ldots, p_{m+1}}(f_1, \ldots, f_{m+1}) := \sum_{Q \in \mathcal{S}} |Q| \prod_{i=1}^{m+1} \langle f_i \rangle_{p_i, Q}, \quad \langle f \rangle_{p,Q} := |Q|^{-1/p} \|f \chi_Q\|_{L^p}.
	\]
	Such expressions dominate quantities $|\langle T(f_1), f_2\rangle|$ for operators \( T \). This type of domination is called \emph{sparse} and plays an important role in Harmonic analysis. For instance, it was used in the proof of the $A_2$ conjecture \cite{lacey_elementary_2017, lerner_simple_2013}. Earlier works related to sparse domination can be found in \cite{hytonen_quantitative_2017, lacey_elementary_2017, lerner_pointwise_2016, li_sparse_2018, volberg_sparse_2018} and the references therein.
	In 2017, Conde-Alonso et al. \cite{conde-alonso_sparse_2017} obtained the following sparse domination for $T$.
	\[
	|(T(f_{1}), f_{2})| \leq \frac{Cp}{p-1} \sup_{S} PSF_{S;1,p}(f_{1}, f_{2}) \left\{ 
	\begin{array}{ll} 
		\|\Omega\|_{L^{r,1}\log L(S^{d-1})}, & 1 < r < \infty, p \geq r'; \\ 
		\|\Omega\|_{L^{\infty}(S^{d-1})}, & 1 < p < \infty. 
	\end{array} 
	\right.
	\]
	As a consequence, they \cite{conde-alonso_sparse_2017} deduced a new sharp quantitative $A_{p}$-weighted estimate for $\mathcal{T}_{\Omega}$. Subsequently, Di Plinio et al. \cite{di_plinio_sparse_2020} provided a sparse bound for the associated maximal truncated singular integrals and certain  quantitative weighted norm estimates were given.
	
	Note that the authors in \cite{culiuc_domination_2018} established a uniform domination of the family of trilinear multiplier forms with singularity over an one-dimensional subspace. Later Barron \cite{barron_weighted_2017} considered the sparse domination for rough bilinear singular integrals with \( \Omega \) in \( L^{\infty}(\mathbb{S}^{2n-1}) \). Recently, Borron's results was  improved by Grafakos, Wang and Xue\cite{grafakos_sparse_2022}.

\begin{thB}[\cite{grafakos_sparse_2022}]\label{thB}
	Let \( \Omega \in L^r (\mathbb{S}^{2n-1}), r > 4/3, \) and \( \int_{\mathbb{S}^{2n-1}} \Omega = 0 \). Let \(\mathcal{T}_{\Omega} \) be the rough bilinear singular integral operator defined in (\ref{1.1}) when \( m=2\ \). Then for \( p > \max \left\{ \frac{24n+3r-4}{8n+3r-4}, \frac{24n+r}{8n+r} \right\} \), there exists a constant \( C = C_{p,n,r} \) such that
	\[
	| \langle \mathcal{T}_{\Omega}(f_1, f_2), f_3 \rangle | \leq C \|\Omega\|_{L^r (\mathbb{S}^{2n-1})} \sup_S PSF_{S; p,p,p)} (f_1, f_2, f_3).
	\]
\end{thB}

The main purpose of this paper is to generalize this sparse bound to the full multilinear setting, extending to the entire integrability range \( r > 1 \) and allowing flexible control via sparse forms with adapted Lebesgue exponents. Our main result is as follows:
	
	\begin{theorem}{\label{th1.1}}
	Let \( 1 < p_1, \ldots, p_m < \infty \) and \( \frac{1}{p} = \sum_{i=1}^{m} \frac{1}{p_i} \) be such that \( \left( \frac{1}{p_1}, \ldots, \frac{1}{p_m} \right) \in \mathcal{H}^m(r) \). Suppose that \( \Omega \in L^r(\mathbb{S}^{mn-1}) \) with \( r > 1 \) and \( \int_{\mathbb{S}^{mn-1}} \Omega(\theta) \, d\sigma(\theta) = 0 \). Then there exist constants \( c = c(n, p_1, \ldots, p_m, r) \) and \( q_i > 1 + \frac{mn}{mn + cr} p_i \) for \( i = 1, \ldots, m+1 \) (with \( p_{m+1} = p' \)), such that there exists a constant \( C = C_{n, q_1, \ldots, q_m, r} \) satisfying
	$$\vert\left<\mathcal{T}_{\Omega}(f_1,\cdots,f_m),f_{m+1}\right>\vert\leq C\Vert\Omega\Vert_{L^{r}(\mathbb{S}^{mn-1})}\sup_{\mathcal{S}}{\rm PSF}_{\mathcal{S};q_1,\cdots,q_{m+1}}(f_1,\cdots,f_{m+1}).$$
\end{theorem}

	In order to state our corollaries, we prepare some background and definitions for certain classes of weights. Recall that, it was Grafakos and Torres \cite{grafakos_multilinear_2002} who initiated the weighted theory for the multilinear singular operators but it was not until 2009 that Lerner et al. \cite{lerner_new_2009} introduced the the multiple  Muckenhuopt weight class, denoted by $A_{\bf p}$, which provides a natural analogue of the linear theory.
	\begin{definition}[Multiple weight class $A_{\bf p}$, \cite{lerner_new_2009}]\label{def1}
		Let $1\leq p_1,\cdots,p_m<\infty$, ${\bf w}=(w_1,\cdots,w_m)$, where each $w_i$ is a nonegative function defined on $\mathbb{R}^n$, and denote $v_{\bf w}=\prod^{m}_{j=1}w_j^{{p}/{p_j}}$. We say ${\bf w}\in A_{\bf p}$ if
		$$[{\bf w}]_{A_{\bf p}}=\sup_{Q}\left(\frac{1}{\vert Q\vert}\int_Q v_{\bf w}(t)dt\right)^{\frac{1}{p}}\prod^{m}_{i=1}\left(\frac{1}{\vert Q\vert}\int_Q w_i^{1-p^{'}_i}(t)\right)^{\frac{1}{p^{'}_i}}<\infty,$$
		where the supremum is taken over all cubes $Q\subset\mathbb{R}^{n}$, and the term $\left(\frac{1}{\vert Q\vert}\int_Q w_i^{1-p^{'}_i}(t)dt\right)^{\frac{1}{p^{'}_i}}$ is understood as $(\inf_Q w_i)^{-1}$ when $p_i=1$, and $p'=p/(p-1)$ is the dual exponent of $p$
	\end{definition}
	 For $m\geq2$, given ${\bf p}=(p_1,\cdots,p_m)$ with $1\leq p_1,\cdots,p_m<\infty$ and ${\bf r}=(r_1,\cdots,r_{m+1})$ with $1\leq r_1,\cdots,r_{m+1}<\infty$, we say that ${\bf r}\prec {\bf p}$ whenever
	$$r_i<p_i,~i=1,\cdots,m~{\rm and}~r_{m+1}^{'}>p,~{\rm where}~\frac{1}{p}:=\frac{1}{p_1}+\cdots+\frac{1}{p_m}.$$
	A more broader class of weights than $A_{\bf p}$ was introduced by Li et al. in \cite{li_extrapolation_2020}.
	\begin{definition}[$A_{\bf p,r}$ weight class, \cite{li_extrapolation_2020}]\label{def2}
		Let $m\geq 2$ be an integer, ${\bf p}=(p_1,\cdots,p_m)$ with $1\leq p_1,\cdots,p_m<\infty$ and ${\bf r}=(r_1,\cdots,r_{m+1})$ with $1\leq r_1,\cdots,r_{m+1}<\infty$. $1/p=\sum^{m}_{k=1}1/p_k$. For each $w_k\in L^{1}_{loc}$, set
		$w=\prod^{m}_{k=1}w_k^{p/p_k}.$
		We say that ${\bf w}=(w_1,\cdots,w_m)\in A_{\bf p,r}$ if $0<w_i<\infty$, $1\leq i\leq m$ and $[w]_{A_{\bf p,r}}<\infty$ with
		$$[w]_{A_{\bf p,r}}=\sup_Q\left(\frac{1}{\vert Q\vert}\int_Q w(x)^{\frac{r^{'}_{m+1}}{r^{'}_{m+1}-p}}dx\right)^{\frac{1}{p}-\frac{1}{r^{'}_{m+1}}}\prod^{m}_{k=1}\left(\frac{1}{\vert Q\vert}\int_Q w_k(x)^{-\frac{1}{\frac{p_k}{r_k}-1}}dx\right)^{\frac{1}{r_k}-\frac{1}{p_k}}.$$
	\end{definition}
	
	By employing sparse domination techniques, we establish some weighted estimates for the operator ${\mathcal T}_{\Omega}$. The first theorem addresses settings involving multiple weights, while the second focuses on the single-weight case.
	
	\begin{corollary}\label{cor1}
		Under the  assumptions as in Theorem \ref{th1.1}, let ${\boldsymbol s}=(s_1,\cdots,s_m),~{\bf q}=(q_1,\cdots,q_{m+1})$. Let
		$\mu_{\bf v}=\prod^{m}_{i=1}v_i^{{s}/{s_k}}$
with $\frac{1}{s}=\sum^{m}_{i=1}\frac{1}{s_i}$ and $1<s<({mn+mnp'+cr})/{mnp'}$. Let ${\boldsymbol s}\prec {\bf q}$ with $s_{m+1}=s'$. Then there is a constant $C=C_{{\bf q},{\boldsymbol s},r,n}$ such that
		\[
		\Vert\mathcal{T}_{\Omega}(f_1,\cdots,f_m)\Vert_{L^{s}(\mu_{\bf v})}\leq C\Vert\Omega\Vert_{L^r}[{\bf v}]^{\max_{1\leq i\leq m}\left\{\frac{q_i}{s_i-q_i}\right\}}_{A_{{\boldsymbol s},{\bf q}}}\prod^{m}_{i=1}\Vert f_i\Vert_{L^{s_i}(v_i)}.\label{cor1.5.1}\\
		\]
	\end{corollary}
	
	\begin{corollary}\label{cor2}
		Under the same assumptions as in Theorem \ref{th1.1}, for $w\in A_{{p}/{2}}$ and $$\max\left\{2,~\max\limits_{1\leq i\leq m+1}\{1+\frac{mn}{mn+cr}p_i\}\right\}<p<\max\limits_{1\leq i\leq m+1}\left\{\frac{2mn+2mnp_i+2cr}{mnp_i}\right\},$$ there exists a constant $C=C_{w,p,n,r}$ such that
		\[
		\Vert \mathcal{T}_{\Omega}(f_1,\cdots,f_m)\Vert_{L^{{p}/{m}}(w)}\leq C\Vert\Omega\Vert_{L^r}\prod^{m}_{i=1}\Vert f_i\Vert_{L^p(w)}.\label{cor1.6.1}
		\]
	\end{corollary}
	
The main idea in the proof of Theorem \ref{th1.1} from \cite{grafakos_sparse_2022} is to elaborate on the decomposition of the rough kernel into smooth kernels with controlled (summable) growth of constants. However, in the case of multilinear operators, more sophisticated interpolation theorems are required. It is necessary to generalize the interpolation in \cite{grafakos_sparse_2022}, which represents the main difficulty to be overcome in this paper.
	
	The article is organized as follows. Section \ref{sec2} contains definitions and basic lemmas. An analysis of the Calder\'{o}n-Zygmund kernel is given in Section \ref{sec3}. Sections \ref{sec4} and \ref{sec5} are devoted to the demonstration of the proof of Theorem \ref{th1.1} and its corollaries. Throughout this paper, the notation \(\lesssim\) will be used to denote an inequality with an inessential constant on the right. We denote by \(\ell(Q)\) the side length of a cube \(Q\) in \(\mathbb{R}^n\) and by \(\text{diam}(Q)\) its diameter. For \(\lambda > 0\) we use the notation \(\lambda Q\) for the cube with the same center as \(Q\) and side length \(\lambda \ell(Q)\).

	\section{Definitions and Main Lemmas}\label{sec2}
	
	Inspired by the techniques in \cite{grafakos_sparse_2022}, we refine their approach to the multilinear singular integral. Firstly, we consider a general multilinear operator that commutes with translations
	\begin{equation}\label{jieduan1}
		{\mathcal T}[K](f_1,\cdots,f_m)(x)={\rm p.v.}\int_{\mathbb{R}^{mn}}K(x-y_1,\cdots,x-y_m) \prod_{i=1}^m f_i(y_i)d\vec{y}
	\end{equation}
	and assume it is bounded multilinear operator mapping $L^{p_1}(\mathbb{R}^n)\times\cdots L^{p_m}(\mathbb{R}^n)\to L^{p}(\mathbb{R}^n)$ for some $1<p_1,\cdots,p_m<\infty$, with $\frac{1}{p}=\sum^{m}_{i=1}\frac{1}{p_i}$. It is assumed that the kernel $K$ of ${\mathcal T}[K]$ has a decomposition of the form
	\begin{equation}\label{jieduan2}
		K(x_1,\cdots,x_m)=\sum_{s\in\mathbb{Z}}K_s(x_1,\cdots,x_m),
	\end{equation}
	where $K_s$ is a smooth truncation of $K$ that enjoys the property
	$${\rm supp}K_s\subset\{(x_1,\cdots,x_m)\in\mathbb{R}^{mn}:2^{s-2}<\vert x_1\vert<2^{s},\cdots,2^{s-2}<\vert x_m\vert<2^{s}\}.$$
	
	The truncation of ${\mathcal T}[K]$ is defined as
	\begin{equation}\nonumber\label{jieduan3}
		{\mathcal T}[K]^{t_2}_{t_1}(f_1,\cdots,f_m)(x):=\sum_{t_1<s<t_m}\int_{\mathbb{R}^{mn}}K_s(x-x_1,\cdots,x-x_m)\prod_{i=1}^m f_i(x_i)d\vec{x},
	\end{equation}
	where $0<t_1<t_2<\infty$. See section 2.1 in \cite{grafakos_sparse_2022} for remarks on this type of truncated operators. In this work, we assume that the truncated norm satisfies
	\begin{equation}\label{CT1}
		\sup_{0<t_1<t_2<\infty}\Vert {\mathcal T}[K]^{t_2}_{t_1}\Vert_{L^{p_1}\times\cdots\times L^{p_m}\to L^{p}}<\infty
	\end{equation}
	for $1<p_1,\cdots,p_m<\infty$, with $\frac{1}{p}=\sum^{m}_{i=1}\frac{1}{p_i}$.
	Since ${\mathcal T}$ is a multilinear operator, we often work with the multilinear form of the type 
	$$<{\mathcal T}(f_1,\cdots,f_m),f_{m+1}>=\int_{\mathbb{R}^{n}}{\mathcal T}(f_1,\cdots,f_m)(x)f_{m+1}(x)dx.$$
	
	In our case, the multilinear truncated form is
	$$<{\mathcal T}[K]^{t_2}_{t_1}(f_1,\cdots,f_m),f_{m+1}>=\int_{\mathbb{R}^n}{\mathcal T}[K]^{t_2}_{t_1}(f_1,\cdots,f_m)f_{m+1}dx.$$
	Denoting by $C_{\mathcal T}(p_1,\cdots,p_m,p)$ the following constants
	\begin{equation}\label{jieduan5}
		C_{\mathcal T}(p_1,\cdots,p_m,p):=\sup_{0<t_1<t_2<\infty}\frac{\vert<{\mathcal T}[K]^{t_2}_{t_1}(f_1,\cdots,f_m),f_{m+1}>\vert}{\Vert f_1\Vert_{L^{p_1}}\cdots\Vert f_m\Vert_{L^{p_m}}\Vert f_{m+1}\Vert_{L^{p'}}}
	\end{equation}
	then (\ref{CT1}) is equivalent to $C_{\mathcal T}(p_1,\cdots,p_m,p)<\infty$.
	\begin{remark}\label{rem}\cite{grafakos_sparse_2022}
		If a multilinear operator of the form (\ref{jieduan1}) is bounded from $L^{p_1}\times\cdots\times L^{p_m}\to L^{p}$ with $p\geq 1$, then so do all of its smooth truncations with kernels
		$$K(x_1,\cdots,x_m)G(x_1/2^{t_1})\cdots G(t_m/2^{t_m})$$
		uniformly on $t_1,\cdots,t_m$. Here $G$ is any function whose fourier transform is integrable.
	\end{remark}
	
	\begin{definition}[Stopping collection \cite{conde-alonso_sparse_2017}]
		Let $\mathcal D$ be a fixed dyadic lattice in $\mathbb{R}^n$ and $Q\in\mathcal{D}$ be a fixed dyadic cube in $\mathbb{R}^n$. A collection $\mathcal{Q}\subset\mathcal{D}$ of dyadic cubes is a stopping collection with top $Q$ if the elements of $\mathcal{Q}$ satisfy
		$$L,~L'\in\mathcal{Q},~L\cap L'\neq\emptyset\Rightarrow L=L'$$
		$$L\in\mathcal{Q}\Rightarrow L\subset 3Q,$$
		and enjoy the separation properties\\
		(i) if $L,~L'\in\mathcal{Q},~\vert s_{L}-s_{L'}\vert\geq 8$, then $7L\cap 7L'=\emptyset$.\\
		(ii) $\cup_{\substack{L\in\mathcal{Q}\\3L\cap 2Q\neq\emptyset}}9L\subset\cup_{L\in\mathcal{Q}}L:=sh\mathcal{Q}$.\\
		Here $s_L=\log_2 l(L)$, where $l(L)$ is the length of the cube $L$.
	\end{definition}
	
	Let ${\bf 1}_{A}$ be the characteristic function of a set $A$. We use $M_p$ to denote the power version of the Hardy-Littlewood maximal function
	$$M_p(f)(x)=\sup_{x\in Q}\left(\frac{1}{\vert Q\vert}\int_Q\vert f(y)\vert^pdy\right)^{\frac{1}{p}},$$
	where the supremum is taken over cubes $Q\subset\mathbb{R}^n$ containing $x$.
	
	We need the following definition.
	\begin{definition}[$\mathcal{Y}_p(\mathcal{Q})~norm$, \cite{conde-alonso_sparse_2017}]
		Let $1\leq p\leq\infty$ and let $\mathcal{Y}_p(\mathcal{Q})$ be the subspace of $L^{p}(\mathbb{R}^{n})$ of the functions satisfying ${\rm supp}~h\subset 3Q$ and
		\begin{equation}\label{norm}
			\infty>\Vert h\Vert_{\mathcal{Y}_p(\mathcal{Q})}:=\begin{cases}
				\max\left\{\Vert h{\bf 1}_{\mathbb{R}^n\backslash sh\mathcal{Q}}\Vert_{\infty},~\sup_{L\in\mathcal{Q}}\inf_{x\in\hat{L}}M_ph(x)\right\},~&p<\infty,\\
				\Vert h\Vert_{\infty},~&p=\infty,\end{cases}
		\end{equation}
		where $\hat{L}$ is the (nondyadic) $2^{5}$-fold dilation of L. We also denote by $\mathcal{X}_p(\mathcal{Q})$ the subspace of $\mathcal{Y}_p(\mathcal{Q})$ of functions satisfying
		$$b=\sum_{L\in\mathcal{Q}}b_L,~{\rm supp}~b_{L}\subset L.$$
		Furthermore, we say $b\in\dot{\mathcal{X}}_{p}(\mathcal{Q})$ if
		$$b\in\mathcal{X}_p(\mathcal{Q}),~\int_{L}b_{L}=0,~\forall L\in\mathcal{Q}.$$
		$\Vert b\Vert_{\mathcal{X}_p(\mathcal{Q})}$ denotes $\Vert b\Vert_{\mathcal{Y}_p(\mathcal{Q})}$ and similar notation for $b\in\dot{\mathcal{X}}_{p}(\mathcal{Q})$. We may omit $\mathcal{Q}$ and simply write $\Vert\cdot\Vert_{\mathcal{X}_p}$ or $\Vert\cdot\Vert_{\mathcal{Y}_p}$.
	\end{definition}
	
	Let $a\land b$ be the minimum of two real numbers $a$ and $b$. Given a stopping collection $\mathcal{Q}$, we define
	\begin{align}\nonumber
		\Lambda_{\mathcal{Q}_{t_1}}^{t_2}(f_1,f_2,\cdots,f_{m+1})&= \nonumber
		\frac{1}{\vert Q\vert}\bigg[<\mathcal{T}[K]^{t_2\land s_{Q}}_{t_1}(f_1{\bf 1}_{Q},f_2,\cdots,f_m),f_{m+1}>\\&\quad -\sum_{\substack{L\in\mathcal{Q}\\L\subset Q}}<\mathcal{T}[K]^{t_2\land s_L}_{t_1}(f_1{\bf 1}_{L},f_2,\cdots,f_{m}),f_{m+1}>\bigg].\nonumber
	\end{align}
	
	Then the support condition
	$${\rm supp}~K_s\subset\{(x_1,\cdots,x_m)\in\mathbb{R}^{mn}:2^{s-2}<\vert x_1\vert<2^{s},\cdots,~2^{s-2}<\vert x_m\vert<2^s\}$$
	gives that
	$$\Lambda_{\mathcal{Q}_{t_1}}^{t_2}(f_1,f_2,\cdots,f_{m+1})=\Lambda_{\mathcal{Q}_{t_1}}^{t_2}(f_1{\bf 1}_{Q},f_2 {\bf 1}_{(m+1)Q},\cdots,f_{m+1}{\bf 1}_{(m+1)Q}).$$
	
	For simplicity, we will often suppress the dependence of $\Lambda_{\mathcal{Q}_{t_1}}^{t_2}$ on $t_1$ and $t_2$ by writing $\Lambda_{\mathcal{Q}(f_1,\cdots,f_{m+1})}=\Lambda_{\mathcal{Q}_{t_1}}^{t_2}(f_1,\cdots,f_{m+1})$, when there is no confusion.
	
	\begin{lemma}[\cite{barron_weighted_2017}]\label{L1}
		Let $\mathcal{T}$ be a multilinear operator with kernel $K$ as the above, such that $K$ can be decomposed as in (\ref{jieduan2}) and suppose that the constant $C_{\mathcal T}$ defined in (\ref{jieduan5}) and $C_{\mathcal T}<\infty$. Also let $\Lambda$ be the (m+1)-linear form associated to $\mathcal{T}$. Assume there exist $1\leq p_1,\cdots,p_m,p_{m+1}\leq\infty$ and some positive constant $C_L$ such that the following estimates hold uniformly over all finite truncations, all dyadic lattices $\mathcal{D}$, and all stopping collections $\mathcal{Q}$:
		\begin{align}\label{psf1}
			\vert\Lambda_{\mathcal{Q}}(b,g_2,g_3,\cdots,g_{m+1})\vert &\leq C_L\vert Q\vert\Vert b\Vert_{\dot{\mathcal{X}}_{p_1}}\Vert g_2\Vert_{\mathcal{Y}_{p_2}}\Vert g_3\Vert_{\mathcal{Y}_{p_3}}\cdots\Vert g_{m+1}\Vert_{\mathcal{Y}_{p_{m+1}}}\\ \nonumber
			\vert\Lambda_{\mathcal{Q}}(g_1,b,g_3,\cdots,g_{m+1})\vert &\leq C_L\vert Q\vert\Vert g_1\Vert_{\mathcal{Y}_{\infty}}\Vert b\Vert_{\dot{\mathcal{X}}_{p_2}}\Vert g_3\Vert_{\mathcal{Y}_{p_3}}\cdots\Vert g_{m+1}\Vert_{\mathcal{Y}_{p_{m+1}}}\\ \nonumber
			\vert\Lambda_{\mathcal{Q}}(g_1,g_2,b,\cdots,g_{m+1})\vert &\leq C_L\vert Q\vert\Vert g_1\Vert_{\mathcal{Y}_{\infty}}\Vert g_2\Vert_{\mathcal{Y}_{\infty}}\Vert b\Vert_{\dot{\mathcal{X}}_{p_2}}\Vert g_4\Vert_{\mathcal{Y}_{p_4}}\cdots\Vert g_{m+1}\Vert_{\mathcal{Y}_{p_{m+1}}}\\ \nonumber
			&\vdots\\ \nonumber
			\vert\Lambda_{\mathcal{Q}}(g_1,g_2,\cdots,g_m,b)\vert &\leq C_L\vert Q\vert\Vert g_1\Vert_{\mathcal{Y}_{\infty}}\Vert g_2\Vert_{\mathcal{Y}_{\infty}}\cdots\Vert g_{m}\Vert_{\mathcal{Y}_{\infty}}\Vert b\Vert_{\dot{\mathcal{X}}_{p_{m+1}}}.
		\end{align}
		Also let $\vec{p}=(p_1,\cdots,p_{m+1})$. Then there is some constant $c_d$ depending on the dimension $d$ such that
		$$\sup_{\mu,\nu}\vert\Lambda^{\nu}_{\mu}(f_1,\cdots,f_m,f_{m+1})\vert\leq c_d[C_{\mathcal T}+C_L]\sup_{\mathcal{S}}{\rm PSF}_{\mathcal{S};\vec{p}}(f_1,\cdots,f_m,f_{m+1})$$
		for all $f_j\in L^{p_j}(\mathbb{R}^n)$ with compact support, where the supremum is taken with respect to all sparse collections $\mathcal{S}$ with some fixed sparsity constant that depends only on $n,~m$.
	\end{lemma}
	
	\begin{remark}[\cite{grafakos_sparse_2022}]\label{T}
		Lemma \ref{L1} is a crucial ingredient of our proof as it implies that
		$$\vert<\mathcal{T}(f_1,\cdots,f_m),f_{m+1}>\vert\leq(C_{\mathcal T}+C_L)\Vert\Omega\Vert_{L^r(\mathbb{S}^{2n-1})}\sup_{\mathcal{S}}PSF_{\mathcal{S};\vec{p}}(f_1,\cdots,f_m,f_{m+1}).$$
	\end{remark}
	
	Next we will consider the interpolation involving $\mathcal{Y}_q$-sparse, of which the precursor can be seen in \cite{grafakos_sparse_2022}. We generalize this to variable parameters.
	
	\begin{lemma}\label{interpolation}
		Let $0<A_2\leq A_1<\infty$, $r>1$, $\vec{\frac{1}{p}}\in\mathcal{H}^m(r)$, $\frac{1}{p}=\frac{1}{p_1}+\cdots+\frac{1}{p_m}$, $0<\epsilon_1,\cdots,\epsilon_{m+1}<1$ and $q_i=1+\epsilon_ip_i$, $i=1,\cdots,m$, $q_{m+1}=1+\epsilon_{m+1}p'$. Suppose that $\Lambda_{\mathcal{Q}}$ is a (sub)-multilinear form such that
		\begin{align}
			\vert\Lambda_{\mathcal{Q}}(f_1,f_2,\cdots,f_m,f_{m+1})\vert &\lesssim A_1 \Vert f_1\Vert_{\dot{\mathcal{X}}_{1}}\Vert f_2\Vert_{\mathcal{Y}_{1}}\cdots\Vert f_m\Vert_{\mathcal{Y}_{1}}\Vert f_{m+1}\Vert_{\mathcal{Y}_{1}},\\
			\vert\Lambda_{\mathcal{Q}}(f_1,f_2,\cdots,f_m,f_{m+1})\vert &\lesssim A_2 \Vert f_1\Vert_{\dot{\mathcal{X}}_{p_1}}\Vert f_2\Vert_{\mathcal{Y}_{p_2}}\cdots\Vert f_m\Vert_{\mathcal{Y}_{p_m}}\Vert f_{m+1}\Vert_{\mathcal{Y}_{p'}}.
		\end{align}
		Then we have
		$$\vert\Lambda_{\mathcal{Q}}(f_1,f_2,\cdots,f_m,f_{m+1})\vert\lesssim A_1^{1-\frac{\widetilde{\epsilon}}{m}\widetilde{p}}A_2^{\frac{\widetilde{\epsilon}}{m}\widetilde{p}}\Vert f_1\Vert_{\dot{\mathcal{X}}_{q_1}}\Vert f_2\Vert_{\mathcal{Y}_{q_2}}\cdots\Vert f_m\Vert_{\mathcal{Y}_{q_m}}\Vert f_{m+1}\Vert_{\mathcal{Y}_{q_{m+1}}},$$
		where $\widetilde{\epsilon}=\min_{1\leq i\leq{m+1}}\{\epsilon_i\}$.
	\end{lemma}
	
	\begin{proof}
		Without loss of generality. We may assume $A_2\leq A_1=1$, and$$\Vert f_1\Vert_{\dot{\mathcal{X}}_{q_1}}=\Vert f_2\Vert_{\mathcal{Y}_{q_2}}=\cdots=\Vert f_m\Vert_{\mathcal{Y}_{q_m}}=\Vert f_{m+1}\Vert_{\mathcal{Y}_{q_{m+1}}}=1.$$ Then it is sufficient to prove $$\vert\Lambda_{\mathcal{Q}}(f_1,f_2,\cdots,f_m,f_{m+1})\vert\lesssim A_2^{\frac{\widetilde{\epsilon}}{m}}\widetilde{p}.$$
		To this purpose, we fix $\lambda>1$, denote $f_{>\lambda}=f{\bf 1}_{\vert f\vert>\lambda}$ and decompose $f_1=b_1+g_1$, where
		$$b_1:=\sum_{L\in\mathcal{Q}}\left((f_1)_{>\lambda}-\frac{1}{\vert L\vert}\int_{L}(f_1)_{>\lambda}\right){\bf 1}_{L}.$$
		
		For $f_2,\cdots,f_{m+1}$, we decompose $f_i=b_i+g_i$, with $b_i:=(f_i)_{>\lambda}$, $i=2,\cdots,m+1$. Then we have
		\begin{align}\nonumber
			\Vert b_1\Vert_{\dot{\mathcal{X}}_{1}}\lesssim\lambda^{1-q_1},~&\Vert g_1\Vert_{\dot{\mathcal{X}}_{1}}\leq\Vert g_1\Vert_{\dot{\mathcal{X}}_{p_1}}\lesssim\lambda^{1-\frac{q_1}{p_1}};\nonumber\\
			\Vert b_2\Vert_{\mathcal{Y}_{1}}\lesssim\lambda^{1-q_2},~&\Vert g_2\Vert_{\mathcal{Y}_{1}}\leq\Vert g_2\Vert_{\mathcal{Y}_{p_2}}\lesssim\lambda^{1-\frac{q_2}{p_2}};\nonumber\\
			&\vdots\label{deco}\\
			\Vert b_m\Vert_{\mathcal{Y}_{1}}\lesssim\lambda^{1-q_m},~&\Vert g_m\Vert_{\mathcal{Y}_{1}}\leq\Vert g_m\Vert_{\mathcal{Y}_{p_m}}\lesssim\lambda^{1-\frac{q_m}{p_m}};\nonumber\\
			\Vert b_{m+1}\Vert_{\mathcal{Y}_{1}}\lesssim\lambda^{1-q_{m+1}},~&\Vert g_{m+1}\Vert_{\mathcal{Y}_{1}}\leq\Vert g_{m+1}\Vert_{\mathcal{Y}_{p'}}\lesssim\lambda^{1-\frac{q_{m+1}}{p'}};\nonumber
		\end{align}
	The demonstrations of these estimates are provided at the end of this lemma. Now we estimate $\vert\Lambda_{\mathcal{Q}}(f_1,f_2,\cdots,f_m,f_{m+1})\vert$ by the sum of the following $2^{m+1}$ terms
		$$\sum_{\substack{h_i\in\{b_i,g_i\}\\i=1,\cdots,m+1}}\vert\Lambda_{\mathcal{Q}}(h_1,h_2,\cdots,h_m,h_{m+1})\vert.$$
		
		Let $p_{m+1}=p',~\widetilde{\epsilon}=\min_{1\leq i\leq{m+1}}\{\epsilon_i\},~\widetilde{p}=\min_{1\leq i\leq{m+1}}\{p_i\},~\widetilde{q}=\min_{1\leq i\leq{m+1}}\{q_i\}$. Then for $i=1,\cdots,m+1$, $1-q_i=-\epsilon_i p_i\leq-\widetilde{\epsilon}\widetilde{p}$, $\lambda^{1-q_i}\leq\lambda^{-\widetilde{\epsilon}\widetilde{p}}$, it holds that
		\begin{align}
			&\sum_{\substack{h_i\in\{b_i,g_i\}\\i=1,\cdots,m+1}}\vert\Lambda_{\mathcal{Q}}(h_1,h_2,\cdots,h_m,h_{m+1})\vert\nonumber\\
			&\lesssim A_2\lambda^{m-\widetilde{\epsilon}\widetilde{p}}+C_{m+1}^{1}\lambda^{-\widetilde{\epsilon}\widetilde{p}}+C_{m+1}^{2}\lambda^{-2\widetilde{\epsilon}\widetilde{p}}+\cdots+C_{m+1}^{m}\lambda^{-m\widetilde{\epsilon}\widetilde{p}}+C_{m+1}^{m+1}\lambda^{-(m+1)\widetilde{\epsilon}\widetilde{p}}\nonumber \\
			&\lesssim\lambda^{-\widetilde{\epsilon}\widetilde{p}}(2^{m+1}+A_2\lambda^m).\nonumber
		\end{align}
	Therefore, if we choose $\lambda^{m}=A_2^{-1}$, then $$\sum_{\substack{h_i\in\{b_i,g_i\}\\i=1,\cdots,m+1}}\vert\Lambda_{\mathcal{Q}}(h_1,h_2,\cdots,h_m,h_{m+1})\vert\lesssim A^{\frac{\widetilde{\epsilon}}{m}\widetilde{p}}.$$
		
	It remains to establish the estimates in~\eqref{deco} for \(b_i\) and \(g_i\). We shall illustrate the argument by proving
	\[
	\|b_1\|_{\dot{\mathcal{X}}_{1}} \lesssim \lambda^{1 - q_1}
	\quad \text{and} \quad
	\|g_1\|_{\dot{\mathcal{X}}_{p_1}} \lesssim \lambda^{1 - \frac{q_1}{p_1}},
	\]
	as the corresponding bounds for \(b_i\) and \(g_i\), with \(i = 2, \dots, m+1\), follow analogously.
	
	From the definition given in~\eqref{norm} and the normalization \(\|f_1\|_{\dot{\mathcal{X}}_{q_1}} = 1\), we observe that
		\begin{align}
			\Vert\sum_{L\in\mathcal{Q}}(f_1)_{>\lambda}\Vert_{\dot{\mathcal{X}}_{1}}
			=&\sup_{L\in\mathcal{Q}}\inf_{x\in\hat{L}}\sup_{x\in\mathcal{Q}}\left(\frac{1}{\vert Q\vert}\int_{S\cap Q}\vert f_1\vert^{1-q_1}\vert f_1\vert^{q_1}\right)\lesssim\lambda^{1-q_1},\nonumber
		\end{align}
		where \(S=\lambda_{f>\lambda}\cap sh\mathcal{Q}\).
		
		Moreover, for the same reason, we have
		\begin{align}
			\sum_{L\in\mathcal{Q}}\left(\frac{1}{L}\int_{L}(f_1)_{>\lambda}{\bf 1}_{L}\right)
			\leq&\lambda^{1-q_1}\sum_{L\in\mathcal{Q}}\left(\frac{1}{\vert L\vert}\int_{L}\vert f_1\vert^{q_1}\right){\bf 1}_{L}\nonumber			\lesssim&\lambda^{1-q_1}\sum_{L\in\mathcal{Q}}\inf_{x\in\hat{L}}\left(M_{q_1}f_1(x)\right)^{q_1}{\bf 1}_{L}\nonumber
			\lesssim\lambda^{1-q_1}.\nonumber
		\end{align}
		
		Now we rewrite 
		$$g_1:=f_1{\bf 1}_{\mathbb{R}^n\backslash sh\mathcal{Q}}+\sum_{L\in\mathcal{Q}}(f_1)_{\leq\lambda}{\bf 1}_{L}+\sum_{L\in\mathcal{Q}}\frac{1}{\vert L\vert}\int_{L}(f_1)_{>\lambda}{\bf 1}_{L}:=I+II+III.$$
		
		Because of \(g_1\in\dot{\mathcal{X}}_{p_1}\), so \(\Vert I\Vert_{\dot{\mathcal{X}}_{p_1}}=0.\) Then we get
		$$\Vert II\Vert_{\dot{\mathcal{X}}_{p_1}}\leq \lambda^{1-\frac{q_1}{p_1}}.$$
		
		In addition, it is easy to see that
		$$III\lesssim\sum_{L\in\mathcal{Q}}\left(\frac{1}{\vert\hat{L}\vert}\int_{\hat{L}}(f_1)_{>\lambda}\right){\bf 1}_{L}\leq\sum_{L\in\mathcal{Q}}{\bf 1}_{L}, \quad \hbox{and\ } \quad \Vert\sum_{L\in\mathcal{Q}}{\bf 1}_{L}\Vert_{\dot{\mathcal{X}}_{p_1}}\leq 1\leq\lambda^{1-\frac{q_1}{p_1}},$$
	which gives that
		$$\Vert g_1\Vert_{\dot{\mathcal{X}}_{p_1}}\lesssim\lambda^{1-\frac{q_1}{p_1}}.$$
		
	\end{proof}
	
\section{Analysis of the Kernel}\label{sec3}

In Section~\ref{sec2}, we examined the generalized kernel $K$. In this section, we focus specifically on rough kernels. For a fixed $\Omega \in L^r(\mathbb{S}^{mn-1})$ with $r > 1$, we analyze the kernel.
\begin{equation}\label{K}
K(x_1,\cdots,x_m)=\frac{\Omega\left((x_1,\cdots,x_m)/\vert(x_1,\cdots,x_m)\vert\right)}{\vert(x_1,\cdots,x_m)\vert^{mn}}.
\end{equation}

We introduce the relevant notation. Define $\Vert[K]\Vert_r$ and $w_{j,r}[K]$ as follows:
\begin{align*}
\|[K]\|_{r} &:= \sup_{s\in\mathbb{Z}} 2^{\frac{msn}{r'}} \Bigl( \| K_s(x_1,\dots,x_m) \|_{L^r(\mathbb{R}^{mn})} \Bigr), \\
w_{j,r}[K] &= \sup_{s\in\mathbb{Z}} 2^{\frac{msn}{r'}} \sup_{\substack{h\in\mathbb{R}^n \\ |h|<2^{s-j-c_m}}} \Bigl( \| K_s(x_1,\dots,x_m) - K_s(x_1+h,\dots,x_m+h) \|_{L^{r}(\mathbb{R}^n)} \Bigr).
\end{align*}

Based on the results in \cite{barron_weighted_2017, grafakos_sparse_2022}, it is known that if the kernel satisfies $\Vert[K]\Vert_{r} < \infty$ and $\sum_{j=1}^{\infty} w_{j,r}[K] < \infty$, then the assumption $(\ref{psf1})$ of Lemma~\ref{L1} is satisfied. However, verifying $\Vert[K]\Vert_r < \infty$ and $\sum_{j=1}^{\infty} w_{j,r}[K] < \infty$ becomes challenging for the kernel $$K(x_1, \dots, x_m) = \Omega\left((x_1, \dots, x_m)/\vert(x_1, \dots, x_m)\vert\right) \vert(x_1, \dots, x_m)\vert^{-mn},$$ where $\Omega \in L^{r}(\mathbb{S}^{mn-1})$ , $r>1$. To address this challenge, we employ the Littlewood-Paley decomposition method. Specifically, we decompose $K = \sum^{\infty}_{j=-\infty} K_j$ and demonstrate that each component $K_j$ satisfies the required properties. 
If $\Omega$ lies in $L^r(\mathbb{S}^{mn-1})$ with $r>1$, then the associated kernel $K$ given by (\ref{K}) is not a multilinear Calder\'on-Zygmund kernel. However, we can decompose it as a sum of Calder\'on-Zygmund kernels. Let $\eta$ be a Schwartz function in $\mathbb{R}^{mn}$ such that $\hat{\eta}(\xi)=1$ when $\vert\xi\vert\leq 1$ and $\hat\eta(\xi)=0$ when $\vert\xi\vert\geq 2$, and let $\hat{\beta}=\hat\eta-\hat\eta(2\dot)$. We decompose the kernel $K$ by setting $K^i(y)=K(y)\hat\beta(2^{-i}y)$ for $i\in\mathbb{Z}$ and
$$\widehat{K^i_j}(y)=\left(\widehat{K^i}(\cdot)\widehat\beta(2^{-j+i}\cdot)\right)(y),~K_j=\sum_{i=-\infty}^{\infty}K_j^i,~j\in\mathbb{Z}.$$

We decompose $$\mathcal{T}_{\Omega}=\sum_{i=-\infty}^{\infty}\sum_{j=-\infty}^{\infty}{\mathcal{T}_{\Omega}}^{i}_j=\sum_{j=-\infty}^{\infty}{\mathcal{T}_{\Omega}}_j$$ accordingly, where $K_j^i$ is the kernel of ${\mathcal{T}_{\Omega}}^{i}_j$ and $K_j$ is the kernel of ${\mathcal{T}_{\Omega}}_j$.

The following lemmas plays a crucial role in our analysis.

\begin{lemma}[\cite{grafakos_sparse_2022}]\label{cz}
	Let $1<p_1,p_2,\cdots,p_m<\infty$ and $\frac{1}{p}=\sum_{i=1}^{m}\frac{1}{p_i},~1<r<\infty,~j\in\mathbb{Z}$. Then for any $0<\epsilon<1$, there is a constant $C_{n,\epsilon}$ such that
	$$\Vert{\mathcal{T}_{\Omega}}_{j}\Vert_{L^{p_1}(\mathbb{R}^n)\times\cdots\times L^{p_n}(\mathbb{R}^n)\to L^{p}(\mathbb{R}^n)}\leq C_{n,\epsilon}\Vert\Omega\Vert_{L^{r}(\mathbb{S}^{mn-1})}2^{\max(0,j)(\epsilon+\frac{mn}{r})}.$$
\end{lemma}

\begin{lemma}[\cite{dosidis_multilinear_2024}]\label{multi}
	Let ${ (\frac{1}{p_1}, \ldots, \frac{1}{p_m}) }\in\mathcal{H}^{m}(r)$, and $r>1$. Then there exists an $c=c({{p_1}, \ldots, {p_m} },r,n)>0$ and $j_0=j_0(m,n)$ such that
	$$\Vert{\mathcal{T}_{\Omega}}_j\Vert_{L^{p_1}(\mathbb{R}^n)\times\cdots\times L^{p_n}(\mathbb{R}^n)\to L^{p}(\mathbb{R}^n)}\lesssim 2^{-cj}\Vert\Omega\Vert_{L^{r}(\mathbb{S}^{mn-1})}, ~if~j\geq j_0.$$
\end{lemma}

\section{The Proof of Theorem \ref{th1.1}}\label{sec4}
We are now in the position to prove Theorem \ref{th1.1}.

\begin{proof}[\textbf{Proof of Theorem $\ref{th1.1}$}]
	By Littlewood-Paley decomposition of the kernel, $\mathcal{T}_{\Omega}$ can be written as
	$$\mathcal{T}_{\Omega}(f_1,\cdots,f_m)(x):=\sum_{j=-\infty}^{\infty}\int_{\mathbb{R}^{mn}}K_j(x-y_1,\cdots,x-y_m)\prod^{m}_{i=1}f_i(y_i)d\vec{y}:=\sum_{j=-\infty}^{\infty}{\mathcal{T}_{\Omega}}_j(f_1,\cdots,f_m)(x).$$
	Given a stopping collection $\mathcal{Q}$ with top cube $Q$, let $\mathcal{Q}_j$ be defined as

	\begin{align*}
	\Lambda_{\mathcal{Q}_{j,t_1}^{t_1}}(f_1,\cdots,f_m) & = 
	\frac{1}{\vert Q\vert}\Big[ \Big<{\mathcal{T}_{\Omega}}[K_j]^{t_2\land s_{Q}}_{t_1}(f_1{\bf 1}_{Q},f_2,\cdots,f_m),f_{m+1}\Big>\\&\quad -\sum_{\substack{L\in\mathcal{Q}\\L\subset Q}} \Bigl<{\mathcal{T}_{\Omega}}[K_j]^{t_2\land s_L}_{t_1}(f_1{\bf 1}_{L},f_2,\cdots,f_{m}),f_{m+1}\Big> \Big].
\end{align*}
	
	For the sake of simplicity, let's denote $\Lambda_{\mathcal{Q}_j}(f_1,\cdots,f_m)=\Lambda_{\mathcal{Q}_j,t_1}^{t_2}(f_1,\cdots,f_m).$
	
	Our proof will be divided into two parts. Each part should satisfy the assumption (\ref{psf1}) of lemma \ref{L1}. We therefore consider these two parts into two steps.
	
	$\bf Step~1.~Estimate~for~j>0.$
	
	Fix $0<\gamma<1$, by Lemma $\ref{cz}$, ${\mathcal{T}_{\Omega}}_j$ is a multilinear Calder\'{o}n-Zygmund operator with kernel $K_j$, and the size and smoothness conditions constant $A_j\leq C_{n,\gamma}\Vert\Omega\Vert_{L^{r}}2^{j(\gamma+mn/r)}$.
	
	Combining the methods in \cite{barron_weighted_2017}, section \ref{sec3}, we know the kernel of ${\mathcal{T}_{\Omega}}_j$ satisfies $\Vert[K_j]\Vert_{p}\lesssim 2^{j(\gamma+mn/r)}<\infty$ for fixed $j\in\mathbb{Z}$. This enables us to use Lemma \ref{cz} and Proposition 3.3 in \cite{barron_weighted_2017} with $A_j\leq C_{n,\epsilon}\Vert\Omega\Vert_{L^{r}}2^{j(\gamma+mn/r)}$ (choose $\beta=1$ and $p=1$). Hence
	$$\vert\Lambda_{\mathcal{Q}_j}(f_1,\cdots,f_m)\vert\lesssim\Vert\Omega\Vert_{L^r(\mathbb{S}^{mn-1})}2^{j(\gamma+mn/r)}\vert Q\vert\Vert f_1\Vert_{\dot{\mathcal{X}}_1}\Vert f_2\Vert_{\mathcal{Y}_{1}}\cdots\Vert f_{m+1}\Vert_{\mathcal{Y}_{1}}.$$
	
	By Lemma \ref{multi}, choosing $(\frac{1}{p_1}, \ldots, \frac{1}{p_m})\in\mathcal{H}^{m}(r),~r>1$, we have
	$$\vert\Lambda_{\mathcal{Q}_j}(f_1,\cdots,f_m)\vert\lesssim\Vert\Omega\Vert_{L^r(\mathbb{S}^{mn-1})}2^{-cj}\vert Q\vert\Vert f_1\Vert_{\dot{\mathcal{X}}_{p_1}}\Vert f_2\Vert_{\mathcal{Y}_{p_2}}\cdots\Vert f_{m+1}\Vert_{\mathcal{Y}_{p^{'}}}.$$
	
	Interpolation via Lemma \ref{interpolation}, it follows that for any $0<\epsilon_1,\cdots,\epsilon_{m+1}<1$ and $q_i=1+\epsilon_ip_i$, $i=1,\cdots,m$, $q_{m+1}=1+\epsilon_{m+1}p'$ so that
	$$\vert\Lambda_{\mathcal{Q}_j}(f_1,\cdots,f_m)\vert\lesssim\Vert\Omega\Vert_{L^r(\mathbb{S}^{mn-1})}2^{j(\gamma+mn/r)(1-\frac{\widetilde{\epsilon}}{m}\widetilde{p})}2^{-cj(\frac{\widetilde{\epsilon}}{m}\widetilde{p})},$$
	where $\widetilde{\epsilon}=\min_{1\leq i\leq{m+1}}\{\epsilon_i\},~\widetilde{p}=\min_{1\leq i\leq{m+1}}\{p_i\}.$
	
	If we choose $\frac{m(\gamma r+mn)}{\widetilde{p}(\gamma r+mn+cr)}<\widetilde{\epsilon}<1$ and $\widetilde{p}>m$. Then there exists a constant $\rho>0$ such that
	$$\vert\Lambda_{\mathcal{Q}_j}(f_1,\cdots,f_m)\vert\lesssim\Vert\Omega\Vert_{L^r(\mathbb{S}^{mn-1})}2^{-\rho j}\Vert f_1\Vert_{\dot{\mathcal{X}}_{q_1}}\Vert f_2\Vert_{\mathcal{Y}_{q_2}}\cdots\Vert f_{m+1}\Vert_{\mathcal{Y}_{q_{m+1}}}.$$
	
	Summing over $j\in\mathbb{Z}^{+}$, we can conclude that
	$$\sum_{j>0}\vert\Lambda_{\mathcal{Q}_j}(f_1,\cdots,f_m)\vert\lesssim\Vert\Omega\Vert_{L^r(\mathbb{S}^{mn-1})}\Vert f_1\Vert_{\dot{\mathcal{X}}_{q_1}}\Vert f_2\Vert_{\mathcal{Y}_{q_2}}\cdots\Vert f_{m+1}\Vert_{\mathcal{Y}_{q_{m+1}}}.$$
	
	By symmetry, it also yields that
	$$\sum_{j>0}\vert\Lambda_{\mathcal{Q}_j}(f_1,\cdots,f_m)\vert\lesssim\Vert\Omega\Vert_{L^r(\mathbb{S}^{mn-1})}2^{-\rho j}\Vert f_1\Vert_{\mathcal{Y}_{q_1}}\Vert f_2\Vert_{\dot{\mathcal{X}}_{q_2}}\cdots\Vert f_{m+1}\Vert_{\mathcal{Y}_{q_{m+1}}};$$
	$$\vdots$$
	$$\sum_{j>0}\vert\Lambda_{\mathcal{Q}_j}(f_1,\cdots,f_m)\vert\lesssim\Vert\Omega\Vert_{L^r(\mathbb{S}^{mn-1})}2^{-\rho j}\Vert f_1\Vert_{\mathcal{Y}_{q_1}}\Vert f_2\Vert_{\mathcal{Y}_{q_2}}\cdots\Vert f_{m+1}\Vert_{\dot{\mathcal{X}}_{q_{m+1}}}.$$
	
	${\bf Step~2.~Estimate~for~j\leq 0}.$
	In a similar way, by Lemma \ref{cz}, Lemma \ref{multi} and Lemma \ref{interpolation}, then summing over $j\leq 0$, one obtains
	$$\sum_{j\leq 0}\vert\Lambda_{\mathcal{Q}_j}(f_1,\cdots,f_m)\vert\lesssim\Vert\Omega\Vert_{L^r(\mathbb{S}^{mn-1})}\Vert f_1\Vert_{\dot{\mathcal{X}}_{q_1}}\Vert f_2\Vert_{\mathcal{Y}_{q_2}}\cdots\Vert f_{m+1}\Vert_{\mathcal{Y}_{q_{m+1}}}.$$
	By symmetry, it also yields that
	$$\sum_{j\leq 0}\vert\Lambda_{\mathcal{Q}_j}(f_1,\cdots,f_m)\vert\lesssim\Vert\Omega\Vert_{L^r(\mathbb{S}^{mn-1})}2^{-\rho j}\Vert f_1\Vert_{\mathcal{Y}_{q_1}}\Vert f_2\Vert_{\dot{\mathcal{X}}_{q_2}}\cdots\Vert f_{m+1}\Vert_{\mathcal{Y}_{q_{m+1}}};$$
	$$\vdots$$
	$$\sum_{j\leq 0}\vert\Lambda_{\mathcal{Q}_j}(f_1,\cdots,f_m)\vert\lesssim\Vert\Omega\Vert_{L^r(\mathbb{S}^{mn-1})}2^{-\rho j}\Vert f_1\Vert_{\mathcal{Y}_{q_1}}\Vert f_2\Vert_{\mathcal{Y}_{q_2}}\cdots\Vert f_{m+1}\Vert_{\dot{\mathcal{X}}_{q_{m+1}}}.$$
	
	Using Theorem {A}, we can find $(\frac{1}{p_1}, \ldots, \frac{1}{p_m})\in\mathcal{H}^{m}(r),~r>1$ and $\frac{1}{p}=\sum_{i=1}^{m}\frac{1}{p_i}$ such that ${\mathcal{T}_{\Omega}}$ maps $L^{p_1}\times\cdots\times L^{p_{m}}\to L^{p}$. In addition, a smooth truncation of the kernel $K$ also gives rise to an operator with a similar bound (see Remark \ref{rem}), thus we have that $C_\mathcal{T}<\infty$ and (\ref{CT1}) is valid. Hence, $\mathcal{T}_{\Omega}$ satisfies Lemma \ref{L1} and this finishes the proof of Theorem \ref{th1.1}.
\end{proof}

\section{Derivation of the Corollaries}\label{sec5}
\begin{proof}[\textbf{Proof of Corollary $\ref{cor1}$}]
	The methods employed here are adapted from \cite{culiuc_domination_2018}, although the weight classes under consideration differ. 
	
	Under the same assumption as in Theorem \ref{th1.1}, we have $q_i>1+\frac{mn}{mn+cr}p_i$, $i=1,\cdots,m+1$. Define $\sigma=v_{\bf w}^{-{\xi'}/{\xi}}$ with $q_i<\xi_i$, $i=1,\cdots,m$, and $q_{m+1}^{'}>\xi$. By Theorem \ref{th1.1}, in conjunction with duality arguments, for any sparse collection $\mathcal S$, it is enough to show that
	\begin{equation}\label{corproof1}
	{\rm PSF}_{\mathcal{S};q_1,\cdots,q_{m+1}}(f_1,\cdots,f_{m+1})\lesssim\prod_{i=1}^{m}\Vert f_i\Vert_{L^{\xi_i}(v_i)}\Vert f_{m+1}\Vert_{L^{\xi'}(\sigma)}
	\end{equation}
	with bounds independent of $\mathcal S$.
	
	Let
	$w_i=v_i^{\frac{q_i}{q_i-\xi_i}},~~i=1,\cdots,m,$
	$w_{m+1}=\sigma^{\frac{q_{m+1}}{q_{m+1}-\xi'}}$
	and
$f_i=g_iw_i^{\frac{1}{q_i}},~~i=1,\cdots,m+1.$
	Then we have
	$$\Vert f_i\Vert_{L^{\xi}(v_i)}=\Vert g_i\Vert_{L^{\xi}(w_i)},~~i=1,\cdots,m, \quad\hbox{and}\quad \Vert f_{m+1}\Vert_{L^{\xi'}(\sigma)}=\Vert g_{m+1}\Vert_{L^{\xi'}(w_{m+1})}.$$
	
	Let $\xi_{m+1}=\xi'$. It follows that
	\begin{align}
	{\rm PSF}_{{\mathcal S};q_1,\cdots,q_{m+1}}(f_1,\cdots,f_{m+1})\nonumber
	&={\rm PSF}_{{\mathcal S};q_1,\cdots,q_{m+1}}(g_1w_1^{\frac{1}{q_1}},\cdots,g_{m+1}w_{m+1}^{\frac{1}{q_{m+1}}})\nonumber\\
	&=\sum_{{Q}\in {\mathcal S}}\left(\prod_{i=1}^{m+1}w_j(E_{Q})^{\frac{1}{\xi_i}}\left(\frac{\langle g_i^{q_i}w_i\rangle_{Q}}{\langle w_i\rangle_{Q}}\right)^{\frac{1}{q_i}}\right)\left(\prod_{i=1}^{m+1}(\langle w_i\rangle_{Q})^{\frac{1}{q_i}-\frac{1}{\xi_i}}\right)\\&\quad\times\left(\vert Q\vert\prod_{i=1}^{m+1}\left(\frac{\langle w_i\rangle_{Q}}{w_i(E_{Q})}\right)^{\frac{1}{\xi_i}}\right).\nonumber
	\end{align}
	
	By a simple calculation, we have
	\begin{align}
	\prod_{i=1}^{m}\langle w_i\rangle_{Q}^{\frac{1}{q_i}-\frac{1}{\xi_i}}\langle w_{m+1}\rangle_{Q}^{\frac{1}{q_{m+1}}-\frac{1}{\xi'}}&=\prod_{i=1}^{m}\langle w_i\rangle_{Q}^{\frac{1}{q_i}-\frac{1}{\xi_i}}\langle v_{\bf w}^{\frac{q_{m+1}^{'}}{q_{m+1}^{'}-\xi}}\rangle_{Q}^{\frac{1}{\xi}-\frac{1}{q_{m+1}^{'}}}\nonumber=[\bf v]_{A_{{\boldsymbol \xi},\bf q}}.\nonumber
	\end{align}
	
	We now deal with the second product  using the technique in \cite{lerner_intuitive_2019}.
	Let
	$$x_i=\frac{q_i-\xi_i}{q_i\xi_i},~~i=1,\cdots,m, \quad\hbox{and\ }x_{m+1}=\frac{q_{m+1}-\xi'}{q_{m+1}\xi'}.$$
	Then $\prod_{i=1}^{m+1}w_i^{-\frac{x_i}{2}}=1.$
	The $\rm H\ddot{o}lder's$ inequality and the fact that
	$$\bigg(\sum_{i=1}^{m+1}\frac{1}{2q'}\bigg)-\bigg(\sum_{i=1}^{m+1}\frac{x_i}{2}\bigg)=1$$
	imply that
	$$\prod_{i=1}^{m+1}\left(w_i(E_{Q})\right)^{-\frac{x_i}{2}}E_{Q}^{\frac{1}{2q'}}\geq\int_{E_{Q}}\prod_{i=1}^{m+1}w_i^{-\frac{x_i}{2}}=\vert E_{Q}\vert.$$
	This together with the sparseness of $\mathcal S$ yields that
	$$\prod_{i=1}^{m+1}\left(\frac{w_i(E_{Q})}{\vert Q\vert}\right)^{-\frac{x_i}{2}}\geq \eta^{\sum_{i=1}^{m+1}\frac{x_i}{2}}.$$
	Therefore
	$$\prod_{i=1}^{m+1}\left(\frac{{\langle w_i\rangle}_{Q}}{\frac{1}{\vert Q\vert}w_i(E_{Q})}\right)^{-\frac{x_i}{2}}\leq\eta^{\sum_{i=1}^{m+1}\frac{x_i}{2}}\prod_{i=1}^{m+1}\langle w_i\rangle^{-\frac{x_i}{2}}_{Q}.$$
	
	By Definition \ref{def1}, we have
	\begin{align}
	\prod_{i=1}^{m+1}\left(\frac{{\langle w_i\rangle}_{Q}}{\frac{1}{\vert Q\vert}}\right)^{\frac{1}{\xi_i}}&\leq\left(\eta^{\sum_{i=1}^{m+1}x_i}\prod_{i=1}^{m+1}{\langle w_i\rangle}_{Q}^{-x_i}\right)^{\max\left(-\frac{1}{x_i\xi_i}\right)}\nonumber\leq\left(\eta^{\sum_{i=1}^{m+1}x_i}[\bf v]_{A_{{\boldsymbol \xi},{\bf q}}}\right)^{\max\left(-\frac{1}{x_i\xi_i}\right)}.\nonumber
	\end{align}
Note that, by \cite{culiuc_domination_2018}, the first product depends on the $L^{\xi_i}(w_i)$-boundedness of $M_{q_i,w_i}$, where
	$$M_{q_i,w_i}f(x)=\sup\limits_{Q\in x}\left(\frac{1}{\vert w(Q)\vert}\int_{Q}\vert f\vert^{q_i}w_i\right)^{\frac{1}{q_i}}.$$
	This concludes the proof of (\ref{corproof1}).
\end{proof}

\begin{proof}[\textbf{Proof of Corollary $\ref{cor2}$}]
	Under the same assumption as in Theorem \ref{th1.1}, in conjunction with duality arguments, we have $q_i>1+\frac{mn}{mn+cr}p_i$, $i=1,\cdots,m+1$. For $2<\xi<\infty$, let $\sigma=w^{-\frac{2}{2-\xi}}$, and $q_i<\rho$, $q_i<\xi$. By Theorem \ref{th1.1} and duality, it is enough to prove that for any sparse collection $\mathcal S$, we have
	$${\rm PSF}_{{\mathcal S};q_1,\cdots,q_{m+1}}(f_1,\cdots,f_{m+1})\lesssim\prod_{i=1}^{m}\Vert f_i\Vert_{L^{\xi}(w)}\Vert f_{m+1}\Vert_{L^{\rho}(\sigma)}$$
	with bounds independent of $\mathcal S$. The proof of this fact is omitted as it follows from the same step as in section 5 in \cite{barron_weighted_2017}.
\end{proof}

Next, we provide another corollary which is related to Corollary 1.7 in \cite{culiuc_domination_2018}.

\begin{corollary}\label{cor3}
	Under the same assumption as in Theorem \ref{th1.1}. For $q_i>1+\frac{mn}{mn+cr}p_i$, $i=1,\cdots,m+1$, $\frac{1}{q}=\sum_{i=1}^{m}\frac{1}{q_i}$. Then for weights $w_i^2\in A_{q_i}$, $i=1,\cdots,m$, $w=\prod_{i=1}^mw_i^{\frac{q}{q_i}}$, there exists a constant $C=C_{w,q_1,\cdots,q_m,n,r}$ such that
	$$\Vert{\mathcal T}_{\Omega}(f_1,\cdots,f_m)\Vert_{L^q(w)}\leq C\Vert\Omega\Vert_{L^{r}({\mathbb S}^{mn-1})}\prod^{m}_{i=1}\Vert f_i\Vert_{L^{p_i}(w_i)}.$$
\end{corollary}

We end this section with another corollary concerning the commutator of a rough ${\mathcal T}_{\Omega}$ with a pair of BMO functions ${\bf b}=(b_1,\cdots,b_m)$. For a pair ${\bf \alpha}=(\alpha_1,\cdots,\alpha_m)$ of nonnegative integers, we define this commutator (acting on a pair of nice functions $f_j$) as follows:
$$[{\mathcal T}_{\Omega},{\bf b}](f_1,\cdots,f_m)(x)=p.v.\int_{{\mathbb R}^{mn}}K(x-y_1,\cdots,x-y_m)\prod_{i=1}^{m}f_i(y_i)\prod_{j=1}^m(b_i(x)-b_i(y_i))^{\alpha_i}d{\vec y}.$$

The following Lemma was given in [\cite{li_extrapolation_2020}, Proposition 5.1], which is crucial for proving the boundedness of the commutator.

\begin{lemma}[\cite{li_extrapolation_2020}]\label{Com}
	Let $T$ be an $m$-linear operator and let $\vec{r} = (r_1, \ldots, r_{m+1})$, with $1 \leq r_1, \ldots, r_{m+1} < \infty$. Assume that there exists $\vec{s} = (s_1, \ldots, s_m)$, with $1 \leq s_1, \ldots, s_m < \infty$, $1 < s < \infty$, and $\vec{r} \prec \vec{s}$, such that for all $\vec{w} = (w_1, \ldots, w_m) \in A_{\vec{s}, \vec{r}}$, we have
	\[
	\|T(f_1, f_2, \ldots, f_m)\|_{L^s(w)} \lesssim \prod_{i=1}^m \|f_i\|_{L^{s_i}(w_i)},
	\]
	where $\frac{1}{s} := \frac{1}{s_1} + \cdots + \frac{1}{s_m}$ and $w := \prod_{i=1}^m w_i^{\frac{s}{s_i}}$.
	
	Then, for all weights $\vec{v} \in A_{\vec{s}, \vec{r}}$, for all $\boldsymbol{b} = (b_1, \ldots, b_m) \in \text{BMO}^m$, and for each multi-index $\alpha$, we have
	\begin{equation}
	\|[T, \boldsymbol{b}]_\alpha(f_1, f_2, \ldots, f_m)\|_{L^s(v)} \lesssim \prod_{i=1}^m \|b_j\|_{\text{BMO}}^{\alpha_i} \|f_i\|_{L^{s_i}(v_i)},
	\end{equation}
	where $v := \prod_{i=1}^m v_i^{\frac{s}{s_i}}$.
\end{lemma}
We now present our result for the commutators as follows.
\begin{corollary}\label{cor4}
	Under the same assumptions as in Theorem \ref{th1.1}, let ${\boldsymbol s}=(s_1,\cdots,s_m),~{\bf q}=(q_1,\cdots,q_{m+1})$ with ${\boldsymbol s}\prec {\bf q}$. Let
	$\mu_{\bf v}=\prod^{m}_{i=1}v_i^{{s}{s_k}}$
	and $\frac{1}{s}=\sum^{m}_{i=1}\frac{1}{s_i}$, $1<s<\frac{mn+mnp'+cr}{mnp'}$, and let $s_{m+1}=s'$. Then there is a constant $C=C_{{\bf q},{\boldsymbol s},r,n}$ such that
	\begin{equation*}
	\Vert [{\mathcal T}_{\Omega},{\bf b}](f_1,\cdots,f_m)\Vert_{L^{s}(\mu_{\bf v})}\leq C\Vert\Omega\Vert_{L^r}[{\bf v}]^{\max_{1\leq i\leq m}\left\{\frac{q_i}{s_i-q_i}\right\}}_{A_{\bf s,q}}\prod_{i=1}^{m}\Vert b_i\Vert_{BMO}^{\alpha_i}\prod^{m}_{i=1}\Vert f_i\Vert_{L^{s_i}(v_i)}.
	\end{equation*}
\end{corollary}
\begin{proof}
	It immediately follows from Lemma \ref{Com} and Corollary \ref{cor1}.
\end{proof}

	\end{document}